\documentclass[12pt,reqno]{amsart}

\usepackage[utf8]{inputenc}
\usepackage[english]{babel}
\usepackage{newlfont}
\usepackage{color}
\usepackage{ amssymb }
\usepackage{graphicx}
\usepackage{amsmath}
\usepackage{scalerel}
\usepackage{stackengine,wasysym}
\usepackage{mathrsfs}
\usepackage{enumitem}
\usepackage{braket}
\usepackage{blindtext}
\usepackage[utf8]{inputenc}
\usepackage{fancyhdr,appendix,psfrag,layout}
\usepackage{hyperref}
\usepackage{url}
\usepackage[capitalise,nameinlink,noabbrev]{cleveref}
\crefname{equation}{}{}
\crefname{enumi}{}{} 
\usepackage{amsfonts}
\usepackage{filecontents}
\usepackage{amsthm}
\usepackage[english]{babel}
\usepackage{fancyhdr}
\addto\captionsenglish{}

\newtheorem{lemma}{Lemma}[section]
\newtheorem{proposition}[lemma]{Proposition}
\newtheorem{theorem}[lemma]{Theorem}

\theoremstyle{definition}

\newtheorem{remark}{Remark}

\numberwithin{equation}{section}

\def\SO{\mathrm{SO}}

\begin{document}

\textwidth=450pt\oddsidemargin=0pt

\title{The 2-rank of finite groups acting on  hyperelliptic 3-manifolds}

\author[M. L. Frisch Sbarra]{Max Leopold Frisch Sbarra}
\address{M. L. Frisch Sbarra: School of Mathematics, University of Birmingham, Watson Building, Edgbaston, Birmingham, B15 2TT, England} \email{mxf237@student.bham.ac.uk}

\author[M. Mecchia]{Mattia Mecchia}
\address{M. Mecchia: Dipartimento Di Matematica e Geoscienze, Universit\`{a} degli Studi di Trieste, Via Valerio 12/1, 34127, Trieste, Italy.} \email{mmecchia@units.it}




\begin{abstract}

We consider 3-manifolds admitting the action of an involution such that its space of orbits is  homeomorphic to $S^3.$ Such involutions are called \textit{hyperelliptic} as the manifolds admitting such an action. We prove that the sectional 2-rank of a finite group acting on a 3-manifold and containing a hyperelliptic involution with  fixed-point set with two components has sectional 2-rank at most four; this upper bound is sharp. The cases where the hyperelliptic involution has a fixed-point set with a number of components different from 2 have been already considered in literature.  Our result completes the analysis and we  obtain some  general results where the number of the components of the fixed-point set is not fixed. In particular, we obtain that a finite group acting on a 3-manifold and containing a hyperelliptic involution has 2-rank at most four, and four is the best possible upper bound.

Finally, we restrict to the basic case of simple groups acting on hyperelliptic  3-manifolds: we use our result about the sectional 2-rank to  prove that a simple group containing  a hyperelliptic  involution is isomorphic to $PSL(2,q)$ for some odd prime power $q$, or to one of four other small simple groups.

\end{abstract}

\maketitle
  
\section{Introduction} 

A \textit{hyperelliptic 3-manifold} is a  3-manifold  which admits the action of a \textit{hyperelliptic involution}, i.e.,  an involution with  quotient space homeomorphic to $S^3$. In this paper  3-manifolds are all smooth, closed and orientable, and the actions of groups on 3-manifolds are  smooth and orientation-preserving.

The analogous concept in dimension 2 is a classical research theme. The hyperelliptic involution  of  a Riemann surface of genus at least two is unique and central in the automorphism group; consequently, the class of automorphism groups of hyperelliptic Riemann surfaces is very restricted (liftings of finite groups acting on $S^2$). 

For 3-manifolds the situation is different: we have many examples of finite groups acting on 3-manifolds containing several non-central hyperelliptic involutions, so in general it is not possible to reconstruct the whole group by lifting  a group acting on $S^3$.

In dimension 3, the study of hyperelliptic involutions is related to the representation of 3-manifolds as cyclic branched covers of links.   In particular, a 3-manifold admits the action of a hyperelliptic involution if and only if the  3-manifold is the 2-fold cover of  $S^3$ branched along a link (in brief \textit{the 2-fold branched cover of a link}).  In fact, the fixed-point set of a hyperelliptic involution is a nonempty union of disjoint, simple, closed curves. If $h$ is a hyperelliptic involution acting on $M$, the projection of  the fixed-point set of $h$ to $M/\langle h \rangle\cong S^3$ is  a link and $M$ is the 2-fold branched cover of this link. On the other hand, if $M$ is a 2-fold branched cover of a link, the involution generating the deck transformation group is hyperelliptic. 


In the present paper we consider the sectional 2-rank of groups containing a hyperelliptic involution.
We recall the following definition: the  \textit{rank} of a group $G$ is the minimal cardinality of a set of  generators of $G;$  the \textit{$2$-rank} of $G$ is the maximum rank of all elementary abelian $2$-subgroups of $G;$ the sectional \textit{$2$-rank} of $G$ is the maximum rank of all $2$-subgroups of $G.$ When $G$ is a 2-group the sectional \textit{$2$-rank} is simply called  \textit{sectional rank}.

Let $G$ be  a finite group acting on a 3-manifold $M$ and containing a hyperelliptic involution $h$ whose fixed-point set has $r$ components.  The first $\mathbb{Z}_2$-homology group  of $M$  has rank $r-1$ (see, for example, \cite[Sublemma 15.4]{sakuma1995homology}). Hence, a hyperelliptic 3-manifold  determines the number of components of the fixed-point set  of its hyperelliptic involutions and  the situation of different numbers of components can be separately analyzed.

If $r=1$  the manifold $M$ is  $\mathbb{Z}_2$-homology sphere. A finite group acting on a $\mathbb{Z}_2$-homology sphere has  2-rank at most 3 and sectional 2-rank  at most 4.  These upper bounds  can be deduced by \cite{Dotzel-Hamrick} or were proved directly in \cite{reni2000} (see also \cite{mecchia2006finite}). Both  are sharp: there are groups acting on the 3-sphere realizing them.   In \cite{reni2000} these results are used  to prove that there are at most nine inequivalent knots with the same hyperbolic 2-fold branched covers. 

The case $r\geq 3$ has been considered in \cite{mecchia2020finite}, where it is proven that the 2-rank is at most 4 (Proposition 4). Moreover, combining \cite[Proposition 3 and Lemma 5]{mecchia2020finite} with  the result about $\mathbb{Z}_2$-homology sphere, we can obtain that $G$ has sectional 2-rank at most six. This case corresponds to the  2-fold branched covers of links with at least three components, the analysis of finite 2-groups acting on $M$ gives that we have at most three such links with the same hyperbolic 2-fold branched cover (see \cite[Corollary 1]{mecchia2020finite}).   

Until now, the case $r=2$ has not be treated in literature from an algebraic point of view: if the fixed-point set of the hyperelliptic involution 
has two components some additional difficulties arise. There are some results regarding the hyperbolic 2-fold branched covers, but they are obtained by a purely geometric approach (see \cite{mecchia-reni2002}). In the present paper we consider this case, and we prove the following theorem about the sectional 2-rank.

\begin{theorem}\label{main}
Let $M$ be a smooth, closed and orientable 3-manifold, and let $G$ be a finite group which acts on $M$ smoothly and orientation-preservingly. If $G$ contains a hyperelliptic involution whose fixed-point set has two components, then the sectional 2-rank of $G$ is at most 4. 

\end{theorem}

There is an example which tells us that our result is sharp. The projective space $\mathbb{R}P^3$ admits an action of the elementary abelian group of rank four containing a hyperelliptic involution. This action is obtained as follows:
\\
Consider the central product 
\begin{align*}
    Q_8\times_{\mathbb{Z}_2}Q_8
\end{align*}
of the quaternion group $Q_8$ with itself (i.e. the direct product of $Q_8$ with itself with identified centers). This group is a subgroup of the special orthogonal group $SO(4)\cong S^3\times_{\mathbb{Z}_2}S^3$ and acts on the 3-sphere, see \cite{mecchia-seppi}. This action on the 3-sphere contains a central free involution (the antipodal map).  The quotient of $S^3$ by the subgroup generated by the antipodal map is the projective space $\mathbb{R}P^3$.   The action induced  by  $ Q_8\times_{\mathbb{Z}_2}Q_8$ on $\mathbb{R}P^3$  is the action of an elementary abelian group of rank four. It contains a hyperelliptic involution whose fixed point set has two components.

The 2-rank is bounded from above by the sectional 2-rank, then this example shows that four is a sharp upper bound also for the 2-rank of $G.$

The case $r=2$ was the missing case, so we can state a general result where the number of components of the fixed-point set is not fixed.

\begin{theorem}\label{main_general}
Let $M$ be a smooth, closed and orientable 3-manifold, and let $G$ be a finite group acting smoothly and orientation-preservingly on $M$.  If $G$ contains a hyperelliptic involution, then the  2-rank of $G$ is at most 4 and the sectional 2-rank is at most six. Moreover, the upper bound for the 2-rank is sharp. 

\end{theorem}

\medskip

We also consider the class of simple groups that has been already studied for the other cases; in general, the analysis of 2-subgroup  is  a key step in the study of simple groups acting on 3-manifolds.

If $G$ is a finite simple group containing  a hyperelliptic involution with connected fixed-point set, then $G$ is isomorphic to the alternating group $\mathbb{A}_5$ \cite{zimmermann2015}. This result is a consequence of a more general one that states that a simple group acting on a $\mathbb{Z}_2$-homology 3-sphere is isomorphic to a linear fractional group $PSL(2,q)$ for some odd prime power  $q$ or to the alternating group $\mathbb{A}_7$ \cite{mecchia2006finite}. The proof of this result  is based on the Gorenstein-Walter classification of simple groups with dihedral Sylow 2-subgroups, a highly non-trivial result.  We remark that each $\mathbb{Z}_2$-homology sphere is  a rational homology 3-sphere, but for this broader class of 3-manifolds the situation changes completely: each finite group acts on a rational homology 3-sphere (see \cite{cooper-long} for free actions and \cite{boileau-et-al} for non-free actions). A complete classification of the finite groups acting on integer or $\mathbb{Z}_2$-homology 3-sphere is not known and appears to be difficult.

The case $r\geq 2$ is considered again in \cite{mecchia2020finite}, where it is proved that  if $G$ is simple, then $G$ is isomorphic to  $PSL(2,q)$ for some odd prime power  $q,$ or to one of  four other small simple groups  ($PSL(3, 5),$ the projective general
unitary group $PSU(3, 3),$ the Mathieu group $M_{11}$ and the alternating group  $\mathbb{A}_7$). The key point of the proof is the existence in $G$ of  a dihedral 2-subgroup weakly closed in the Sylow 2-subgroup;  classifications of simple groups with dihedral Sylow 2-subgroups (by Gorenstein Walter) and with quasi-dihedral Sylow 2-subgroups (by Alperin, Brauer, Gorenstein) are also used.
In the present case, we use the classification of simple groups of sectional 2-rank at most four by Gorenstein and Harada (see \cite[Theorem 8.12]{suzuki1986group}). We obtain the following result:

\begin{theorem}\label{simple}
Let $M$ be a smooth, closed and orientable 3-manifold, and let $G$ be a finite simple group acting smoothly and orientation-preservingly on $M$. If $G$ contains a hyperelliptic involution whose fixed-point set has two components, then $G$ is isomorphic to  the projective special linear group $PSL(2,q)$ for some odd prime power  $q.$
\end{theorem}

Considering the previous cases, we obtain the following result in the general case.

\begin{theorem}\label{simple_general}
Let $M$ be a smooth, closed and orientable 3-manifold, and let $G$ be a finite simple group acting smoothly and orientation-preservingly on $M$.  If $G$ contains a hyperelliptic involution, then $G$ is isomorphic to  the projective special linear group $PSL(2,q)$ for some odd prime power  $q,$ or to one of the following four groups: the projective special linear group $PSL(3, 5),$ the projective general
unitary group $PSU(3, 3),$ the Mathieu group $M_{11}$ and the alternating group  $\mathbb{A}_7.$
\end{theorem}
Considering that each finite group admits an action on a 3-manifold, even on a rational homomolgy 3-sphere, the list of simple groups we obtain is quite restricted. However, to the present day $\mathbb{A}_5$ is the only known simple group admitting an action on a 3-manifold and containing a hyperelliptic involution. 

It is worth mentioning that   examples of groups containing  a hyperelliptic involution are difficult to construct. The straightforward method consists of taking a finite group of  symmetries of a  link and lifting it to the 2-fold branched cover of the link. In this way we obtain groups with a central hyperelliptic involution. By the geometrization of 3-manifolds,   a finite group acting on $S^3$ is conjugate to a finite subgroup of $\SO4$, so we have a very restricted list of possibilities (see \cite{conway-smith, mecchia-seppi}).

Other examples can be obtained by using Seifert spherical 3-manifolds. 
In particular, the  homology Poincar\'e sphere $S^3/I^*_{120},$ the octahedral manifold  $S^3/I^*_{48}$ and  the quaternionic manifold $S^3/Q_8$ admit an action of $\mathbb{A}_5$  that includes a hyperelliptic involution; see \cite{mccullough} for the notations and the computation of the isometry groups, and \cite{mecchia-seppi} for a method to compute the space of orbits of the involutions.  The numbers of components of the fixed-point set of the hyperelliptic involutions in the three examples are 1, 2 and 3, respectively.

The results in the present paper raise naturally the following two questions:
\begin{enumerate}

\item In Theorem \ref{main_general},   is the upper bound of six  for the sectional 2-rank sharp?

\item Which simple  groups in  the list of Theorem \ref{simple_general} admit an action on a 3-manifolds such that a hyperelliptic involution is contained in the group?

\end{enumerate}


\subsection*{Organization of the paper} 

In Section~\ref{preliminaries}, we introduce the necessary preliminaries; Section~\ref{main proof} is devoted to the proof of  Theorem~\ref{main}; in Section~\ref{simple groups} we consider the finite simple group case.



\section{Preliminaries results} \label{preliminaries}

Let $M$ be a smooth, closed and orientable 3-manifold, and let $G$ be a finite group which acts on $M$. The action of $G$ on $M$ is smooth and orientation preserving. The fixed-point set of $g\in G$, which will be denoted by $Fix(g)$, either is empty or is a disjoint union of finitely many simple closed curves.

\begin{proposition}\label{proposition: 2} 

Let $K$ be a simple closed curve in $M$ and $D$ be subgroup of $G$ fixing $K$ set-wise. Then $D$ contains $H$ an abelian subgroup of rank at most two and index at most two. If the index of $H$ in $D$ is  two, $D$ is a semidirect product of $H$ and a subgroup of order two acting dihedrally on $H$ (i.e. a \textit{generalized dihedral group} of rank at most 2).


\end{proposition} 

\begin{proof} 

We may assume that $D$ acts by isometries in some riemannian metric on M. Moreover $K$ has a $D-$invariant tubular neighbourhood of $K,$ and we can restrict this bundle to a unit disk bundle $N$ isometric to the solid torus $S^1\times D^2$. We get that $D$ acts by isometries on the solid torus and,  by Newmann's theorem \cite[p. 157, Theorem 9.5]{bredon1972introduction}, the action is faithful.

Let $H$ be the subgroup of $D$ acting as rotations on the $0-$section of N (the core of the solid torus corresponding to $K$) and $H'$ be the subgroup acting trivially on the $0-$section. We may recall that we are supposing that the group action is orientation preserving and thus it follows that $H$ is abelian. Since $H'$ is cyclic and the quotient group $H/H'$ is cyclic, $H$ is abelian of rank at most two. If $H=D$ the proof is complete, otherwise $H$ has index two in $D$ and the elements of $D$ not contained in $H$ have order two.  They act as reflections on the core and fix two of its points. Moreover, the action of these  involutions on $H$ is dihedral. 

\end{proof}

It is desirable to make a few additional comments on this result and its proof: When we talk of the elements of $H$ we say they act as rotations on $K$ or they are also called $K-$rotations. While for the elements of $D$ not in $H$ we say that they act as reflections on $K$ or we say $K-$reflections.  Thus, all elements of a group fixing $K$ set-wise act on it either by rotation or reflection. The reflections fix pointwise two elements of $K$ and have thus nonempty fixed-point set. Finally, since $H'$ is cyclic there cannot be two distinct involutions fixing $K$ pointwise. 

\medskip

Involutions acting on $\mathbb{Z}_2-$homology  play a key role in the proof of Theorem \ref{main}. We recall  that the cyclic branched cover of $S^3$ along a knot is a $\mathbb{Z}_2-$homology sphere, see \cite[Sublemma 15.4]{sakuma1995homology}.

\begin{remark}

By classical Smith theory \cite[p.148, Theorem 8.1]{bredon1972introduction} the fixed-point set of an involution acting on a $\mathbb{Z}_2-$homology sphere is connected (might be empty). 

\end{remark}

\begin{lemma}\label{lemma MZ}
 If $G\cong \mathbb{Z}_2\times \mathbb{Z}_2$ and $M$ is a $\mathbb{Z}_2-$homology sphere, then either (i) or (ii) holds:

\begin{enumerate}[label=(\roman*)]
\item There is one free involution in $G$ and two involutions with nonempty connected fixed-point sets. The union of the  two fixed-point sets gives two disjoint circles.
\item All three involutions in $G$ have nonempty connected fixed-point set. The union of the three fixed-point sets gives a theta curve  (three circles  intersecting in exactly two points).
\end{enumerate}

\end{lemma}

The proof of the Lemma \ref{lemma MZ} can be found in  \cite[Lemma 2]{mecchia2006finite}

\begin{lemma}\label{lemma: 1} 

Let $h$ be an involution of $G$ which has a fixed-point set with two components and suppose that $M/\langle h\rangle$ is a $\mathbb{Z}_2$-homology 3-sphere. If $f$ is an involution of $G$ commuting with $h$ then the fixed-point set of $f$ has at most two components. 

\end{lemma} 

\begin{proof} 


Let $\pi:M\longrightarrow M/\langle h\rangle$ be the projection. Let $Fix(h)=K_1\cup K_2$  and  $Fix(f)=\cup_{i\in\mathcal{I}}\Tilde{K}_i$ be the fixed-point set of $h$ and $f$ respectively, where $\{K_1,K_2\}$ and $\{\Tilde{K}_i\}_{i\in\mathcal{I}}$ are the components.  Since $f$ and $h$ commute, $f$ fixes $Fix(h)$ set-wise, while $h$ fixes set-wise both $Fix(f)$ and $Fix(fh).$ We denote by $\overline{f}$ the projection of $f$ to $M/\langle h\rangle.$ Since $\overline{f}$ acts on a $\mathbb{Z}_2$-homology sphere,  we have that $Fix(\overline{f})$ is connected (might be empty). Clearly if $Fix(\overline{f})=\pi(Fix(f)\cup Fix(fh))$ is empty, so is $Fix(f).$ For what follows we may suppose that $Fix(\overline{f})$ is a simple closed curve and let it denote by $\overline{K}.$

\medskip
 
Suppose first that the fixed-point sets of $h$ and $f$ have empty intersection. As said $h$ fixes set-wise $Fix(f)$ and being a diffeomorphism, we have for fixed $i$ in $\mathcal{I}$ that either $h(\Tilde{K_i})=\Tilde{K}_i$ or $h(\Tilde{K_i})=\Tilde{K}_j$ for some $j\in\mathcal{I}$ different from $i.$  Moreover since $Fix(f)\cap Fix(h)=\emptyset$ we have that $Fix(f)\cap Fix(fh)=\emptyset.$  This implies that either $Fix(f)$ or $Fix(fh)$ is empty, since otherwise $\pi(Fix(f))$ and $\pi(Fix(fh))$ give a separation of the connected set $\overline{K}.$  We can suppose that $Fix(fh)$ is empty and each  $\pi(\Tilde{K}_i)$ is a simple closed curve. 

If  $h(\Tilde{K}_i)=\Tilde{K}_i$ for all $i\in \mathcal{I}$, then  $\pi(K_1)$ and $\pi(Fix(f)\setminus K_1),$  give a separation of the connected set $\overline{K}$ unless $\pi(Fix(f)\setminus K_1)$ is empty, thus necessarily  $Fix(f)$ is connected. 

If $h(\Tilde{K}_i)=\Tilde{K}_j$ for some $i,j\in\mathcal{I},$ i different from j, we consider $\pi(\Tilde{K_i}\cup\Tilde{K}_j)$ and $\pi(Fix(f)\setminus (\Tilde{K_i}\cup\Tilde{K}_j)),$ again this two give a separation of $\overline{K}$ unless $\pi(Fix(f)\setminus (\Tilde{K_i}\cup\Tilde{K}_j))$ is empty and thus $Fix(f)$ has two components. 


\medskip

Suppose secondly that $Fix(f)$ and $Fix(h)$ have nonempty intersection. In this case we have that the fixed-point set of $fh$ is nonempty.  We may observe that if $K_1$ meets $Fix(f)$ then they intersect in exactly two points. Indeed if they meet then $f$ fixes $K_1$ set-wise and must act as a reflection on $K_1.$ The same argument may be applied to $K_2.$  Therefore, we have that $Fix(f)$ can meet $Fix(h)$ either in two or four points, depending on whether it intersects one component of $Fix(h)$ or both. $Fix(fh)$  meets both $Fix(f)$ and $Fix(h)$ in two or four points.

Let us suppose first that $Fix(f)$ and $Fix(h)$ meet in two points. We have necessarily that a component of $Fix(f)$ and a component of $Fix(h)$ meet in two points and from this it follows that $Fix(f)$ has one component otherwise we obtain a separation of  $\overline{K}$.

Therefore, we may suppose that $Fix(f)$ and $Fix(h)$ meet in exactly four points, consequently also $Fix(fh)$ meets both $Fix(f)$ and $Fix(h)$ in four  points. Let $Fix(fh)=\cup_{j\in\mathcal{J}}K_j'$. 

We consider first the case where $K_1$ meets a component of $Fix(f)$ in two points.  We have that $K_2$ must meet another component  in two points. Suppose for simplicity that  $K_1$ (resp. $K_2$) intersects $\Tilde{K}_1$ (resp. $\Tilde{K}_2$). 
Thus, we have two possibilities.

\begin{figure}[h] 

    \centering 

    \includegraphics[width=0.6\textwidth]{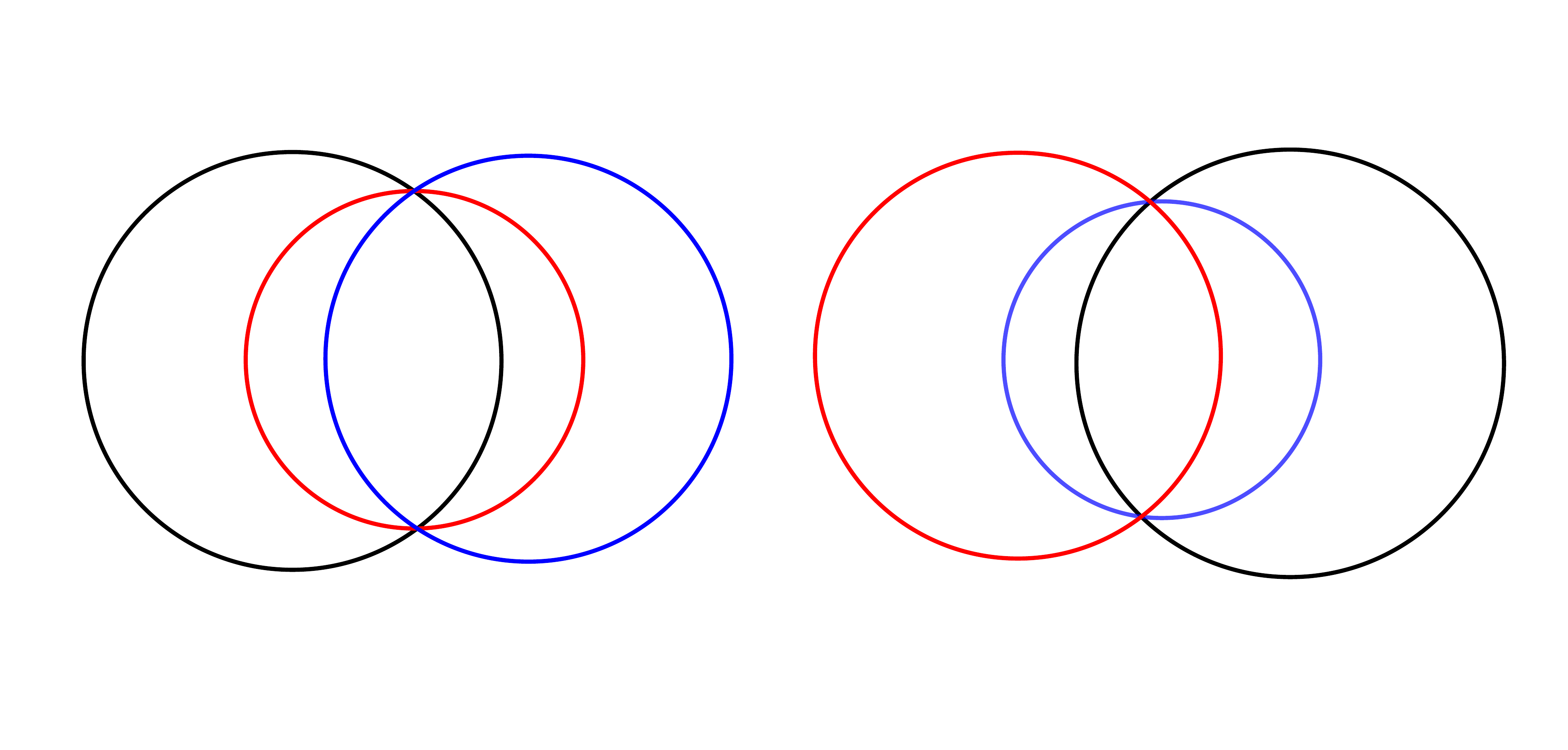} 

    \caption{First possibility. The fixed-point set of $h$ (red), $f$ (blue) and $fh'$ (black).} 

   \label{fig:linearAD} 

\end{figure} 

\begin{figure}[h] 

   \centering 

   \includegraphics[width=0.9\textwidth]{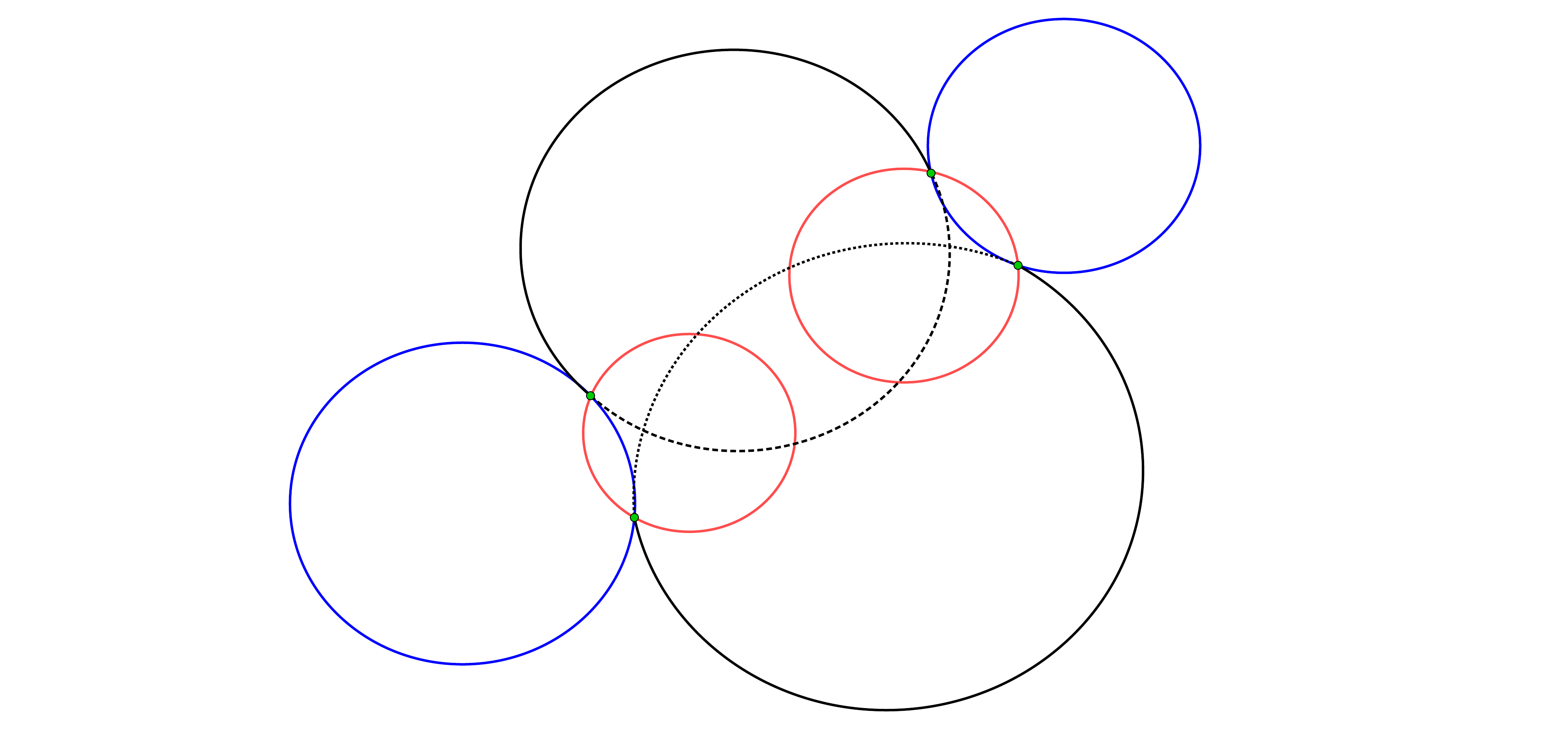} 

    \caption{Second possibility. The fixed-point set of $h$ (red), $f$ (blue) and $fh$ (black).} 

    \label{fig:linearAd1} 

\end{figure} 

The first one is where a component $K_1'$ meets $K_1$ and $\Tilde{K_1}$ in 2 points and a second component $K_2'$ meets $K_2$ and $\Tilde{K}_2$ in 2 points, see \cref{fig:linearAD}, to be more explicit: 
$$ K_1'\cap K_1\cap \Tilde{K}_1=\{P_1,\ P_2\} \quad  K_2'\cap K_2\cap \Tilde{K}_2=\{P_3, \ P_4\} $$
The second possibility is where both $K_1'$ and $K_2'$ meet $K_1$ and $\Tilde{K_1}$ in one point and they meet also $K_2$ and $\Tilde{K_2}$ in point, see \cref{fig:linearAd1}, to be precise 
\begin{align*} 
& K_1'\cap K_1\cap \Tilde{K}_1=P, \ \ K_2'\cap K_1\cap \Tilde{K}_1=P', \\ 
& K_1'\cap K_2\cap \Tilde{K}_2=\Tilde{P}, \ \  K_2'\cap K_2\cap \Tilde{K}_2=\overline{P}. 
\end{align*} 
The first possibility is impossible since it leads to a separation of $\overline{K}$ in similar fashion as already discussed.

Consider the second possibility. We can observe that  $Fix(f)$ has two components since any additional component would lead to a separation of $\overline{K}$ and the same holds for $Fix(fh).$
To complete the proof a similar argument can be applied when $K_1$ intersects two components of $Fix(f)$, each in a single point.

\end{proof}

\begin{proposition}\label{centralizer}

Let $S$ be a subgroup of $G$ and  $g\in G$ be an element whose  fixed-point set has two components.  Suppose that $S$  centralizes $g$, then  the following hold
\begin{enumerate}[label=(\roman*)]

\item if $p$ is an odd prime the Sylow $p$-subgroup of $S$ is normal in $S$;

\item $S$ is generated by four elements;

\item $S$ is solvable.

\end{enumerate}

\end{proposition} 

\begin{proof} 

Let $K_1$ be one of the two components of the fixed-point set of $g$. Since $g$ is centralized by $S,$ all elements of $S$ fix $Fix(g)$ set-wise.   

Consider $H_3$ the subgroup generated by the elements of $S$ which fix $K_1$ set-wise.  Since $g$ has two components, $H_3$ is a normal subgroup of $S$ of index at most 2, which contains all elements of odd order of $S$. By Proposition \ref{proposition: 2}, $H_3$ is abelian or generalized dihedral group of rank at most 2. In any case, if we denote by  $H_2$ the subgroup of $K_1$-rotations, we have that $H_2$ is a normal abelian subgroup of $H_3$ of index at most 2, containing all the elements of odd order.  The rank of $H_2$ is at most two and this implies that there exists a set of generators with at most four elements. If $p$ is odd any $p$-subgroup of $S$ is contained in $H_3$ that has a unique Sylow $p$-subgroup, this implies the first item. By the  subnormal series $H_2 \vartriangleleft H_3 \vartriangleleft S$ we can obtain that $S$ is solvable.

\end{proof}

\begin{remark}
We consider the subgroups $H_3$  and $H_2$ introduced in the proof of the proposition. If we define also $H_1$ as the subgroup fixing pointwise the curve $K_1$ we obtain  the following subnormal series 
\begin{align*} 
H_1\vartriangleleft H_2 \vartriangleleft H_3 \vartriangleleft S, 
\end{align*} 
where $H_1$ is cyclic, $H_2/H_1$ is cyclic, $H_2$ has index at most 2 in $H_3$ and $H_3$ has index at most two in $S$.  Keeping in mind the properties  of this subnormal series will be usefull in the proof of Theorem \ref{main}. 
The elements that  exchange the two components of $Fix(g)$ will be called \textit{exchange elements}.

\end{remark}

The following Lemma is simply a  corollary of the Proposition \ref{centralizer} but in this form it is extensively used in the proof of Theorem  \ref{main}

\begin{lemma}\label{lemma: 2} 

If a 2-subgroup $S$ of $G$ contains in its center an element whose  fixed-point set has two components, then then the sectional rank of S is at most 4. 

\end{lemma}


\section{Proof of Theorem \ref{main} } \label{main proof}

We come now the to the main result which is that the sectional 2-rank of $G$ is at most 4, i.e., every 2-subgroup is generated by at most four elements.  

Assume $S$ is a 2-subgroup of $G$ containing a hyperelliptic involution which has fixed-point set with two components.

\begin{remark}\label{Remark: 1} 
We can suppose that $S$ does not contain any involution with nonempty connected fixed-point set otherwise we can conclude by using a result due to Reni and Zimmerman, see \cite[Lemma 2.2]{reni2001finite} for the proof. 
\end{remark} 
\begin{remark}\label{Remark: 2} 
By Lemma \ref{lemma: 1} and the previous remark we can suppose any involution commuting with $h$ in $S$ has either fixed-point set with two components or empty one. 
\end{remark} 

To be able to prove Theorem \ref{main} we need the following proposition; the proof is discussed in \cite{macwilliams19702}: 

\begin{proposition} 

A 2-group with no elementary abelian normal subgroup of rank 3 has sectional rank at most 4. 

\end{proposition} 

In addition, we will use the following result due to Harada, see \cite[p. 514, (8.14)]{suzuki1986group}: 

\begin{proposition} 
Let A be an elementary abelian 2-subgroup of S, suppose that A has rank 3 and $C_S(A)$ is equal to $A$, then the sectional rank of S is equal to 3.  
\end{proposition} 
Thanks to the previous propositions we may suppose that there exists $A$ an elementary abelian normal subgroup of $S$ of rank 3. In addition, we can suppose that $A$ is a proper subgroup of its centralizer $C_S(A)$. Moreover, since for all $g\in S$ we have $C_S(A)^g=C_S(A^g),$ we get that $C_S(A)$ is normal. \newline
\\ 
\textbf{Case 1} Suppose that $A$ contains an involution $h$ such that the fixed-point set has two components $K_1$ and $K_2$, and                 $M/\langle h\rangle$ is a $\mathbb{Z}_2$-homology 3-sphere. We notice that if $A$ contains a hyperelliptic involution we are in Case 1. \newline 
\\ 
\textbf{Case 1.1} Suppose that $A$ contains a $K_1$-reflection.

\smallskip


Let $D$ be the subgroup of $C_S(A)$ fixing $K_1$ set-wise and let $t$ be the $K_1$-reflection contained in $A$. Thus, we have that $ t\Tilde{h}t=\Tilde{h}^{-1}$ for all $\Tilde{h}\in\Tilde{H},$ where $\Tilde{H}$ is the  abelian subgroup of $D$ of $K_1$-rotations. Moreover, since $t$ is in $A$, all elements of $\Tilde{H}$ commute with $t,$ which is possible only if $\Tilde{H}$ is an elementary abelian 2-group of rank at most 2 and $D$ is an elementary subgroup of rank at most 3. 

We have that any element of $C_S(A)$ fixes $Fix(h)$ set-wise. Thus, any element not in $D$ is an exchange element and hence it follows that $D$ has index 2 in $C_S(A).$  This implies that $D$ has rank three, $C_S(A)$ has order 16 and  $A$ has index 2 in $C_S(A)$.

It is a simple verification that in this case $C_S(A)$ is an abelian group and since it has a subgroup isomorphic to $\mathbb{Z}_2^3$ it is either isomorphic to $\mathbb{Z}_2^4$ or to $\mathbb{Z}_2\times\mathbb{Z}_2\times\mathbb{Z}_4.$ \newline 
\\
\textbf{Case 1.1.1} We consider $C_S(A)\cong\mathbb{Z}_2\times\mathbb{Z}_2\times\mathbb{Z}_4;$ in this case  $A$ is the unique elementary subgroup of $C_S(A)$ of rank 3.
\smallskip

It is easy to see that $S/C_S(A)$ is isomorphic to a subgroup of $Aut(\mathbb{Z}_2^3)\cong GL(3,2)$. The order of $GL(3,2)$ is  $(2^3-1)(2^3-2)(2^3-2^2)=3\times 7\times 8.$ The subgroup of upper triangular matrices in $GL(3,2)$   is isomorphic to $\mathbb{D}_8,$ the dihedral group of order 8,  and is a Sylow 2-subgroup. 

Now let $S'$ be a subgroup of $S,$ thus $S/\left(C_S(A)\cap S'\right)$ is isomorphic to $S'C_S(A)/C_S(A),$ which is a subgroup of $S/C_S(A).$ If $C_S(A)\cap S'$ has rank at most 2, since $S'C_S(A)/C_S(A)$ has rank at most 2, we get that $S'$ can be generated by four elements. 

Hence, we may suppose that $C_S(A)\cap S'$ has rank 3 and thus is isomorphic either to $\mathbb{Z}_2^3$ or $\mathbb{Z}_2\times\mathbb{Z}_2\times\mathbb{Z}_4.$ 
We can observe that in both cases $A$ is a subgroup of $S'.$ Indeed if $C_S(A)\cap S'$ is isomorphic to $\mathbb{Z}_2^3$ we have that $C_S(A)\cap S'$ is equal to $A,$ since $C_S(A)$ has a unique subgroup isomorphic to $\mathbb{Z}_2^3,$ and in the case $C_S(A)\cap S'$ is isomorphic to $\mathbb{Z}_2\times\mathbb{Z}_2\times\mathbb{Z}_4$ then clearly $C_S(A)\cap S'$ is equal to $C_S(A)$.

Consider the upper triangular matrices of $GL(3,2):$ 
\begin{align*} 
T=\left\{ 
\Big(\begin{smallmatrix} 
1 & a & b\\ 
0 & 1& c\\ 
0 & 0 & 1 
\end{smallmatrix}
\Big) \big| \ a, \ b,\ c \in \mathbb{Z}_2\right\}. 
\end{align*}

Notice that $T$ is generated by $v=\Big(\begin{smallmatrix} 
1 & 0 & 1\\ 
0 & 1& 1\\ 
0 & 0 & 1 
\end{smallmatrix}\Big)$ and $u=\Big(\begin{smallmatrix} 
1 & 1 & 1\\ 
0 & 1& 1\\ 
0 & 0 & 1 
\end{smallmatrix}\Big). $ 
 
Moreover, we have that 
\begin{align*} 
u^2=\Big(\begin{smallmatrix} 
1 & 0 & 1\\ 
0 & 1& 0\\ 
0 & 0 & 1 
\end{smallmatrix}\Big), \ u^3=\Big(\begin{smallmatrix} 
1 & 1 & 0\\ 
0 & 1& 1\\ 
0 & 0 & 1 
\end{smallmatrix}\Big). 
\end{align*} 

From this it follows that the element $(0,0,1)$ has an orbit containing four elements under the action of $T.$ In addition $T$ has two subgroups isomorphic to $\mathbb{Z}_2\times\mathbb{Z}_2$:  
\begin{align*} 
T_1 &=\{1,\ u^2,\ v,\ u^2v\}= \left\{1,\ 
\Big(\begin{smallmatrix} 
1 & 0 & 1\\ 
 0 & 1& 0\\ 
0 & 0 & 1 
\end{smallmatrix}\Big) ,\ \Big(\begin{smallmatrix} 
1 & 0 & 1\\ 
0 & 1& 1\\ 
0 & 0 & 1 
\end{smallmatrix}\Big)  ,\ 
\Big(\begin{smallmatrix} 
1 & 0 & 0\\ 
0 & 1& 1\\ 
0 & 0 & 1 
\end{smallmatrix}\Big) \right\} \\ 
T_2 &=\{1,\ u^2,\ uv,\ u^3v\}=\left\{1,\ 
\Big(\begin{smallmatrix} 
1 & 0 & 1\\ 
0 & 1& 0\\ 
0 & 0 & 1 
\end{smallmatrix}\Big) ,\ 
\Big(\begin{smallmatrix} 
1 & 1 & 1\\ 
0 & 1& 0\\ 
0 & 0 & 1 
\end{smallmatrix}\Big), \ 
\Big(\begin{smallmatrix} 
1 & 1 & 0\\ 
0 & 1& 0\\ 
0 & 0 & 1 
\end{smallmatrix}\Big)   \right\} 
\end{align*} 
With an easy calculation we can see that under the action of $T_1$  the orbit of $(0,0,1)$ has four elements.  The action of $T_2$ instead gives in $\mathbb{Z}_2^3$ three orbits of two elements and two orbits of one element.


\smallskip

In the following  we consider the action of $S'/\left(C_S(A)\cap S'\right)$ on $A.$
 
Suppose that $\overline{S'}=S'/\left(C_S(A)\cap S'\right)$ is isomorphic to $\mathbb{D}_8,$ therefore there exist $r$ and $x$ in $S'$ such that $r\left(C_S(A)\cap S'\right)$ and $x\left(C_S(A)\cap S'\right)$ generate $S'/\left(C_S(A)\cap S'\right)$ and have order 4 and 2 in the quotient respectively.

If $C_S(A)\cap S'$ is equal to $A$ by what seen about $T$ there exists an element $y$ of $A$ which has an orbit of length 4 and this orbit contains necessarily a set of  generators of $A.$ This implies that $r,\ x$ and $y$ generate $S'.$

If instead we suppose that $C_S(A)\cap S'$ equal to $C_S(A)$. We have an element $y$ in $A$ with an orbit of length 4. In addition, if we take $f$ an element of order 4 in $C_S(A), $ then we get that the orbit of $y$ and $f$ generate $C_S(A).$  It follows that $r,\ x,\ y, \ f$ generate $S'.$

If the action of  $\overline{S'}$ on $A$ has an orbit containing 4 involutions, we can use the same argument.  The two remaining cases are  $\overline{S'}\cong \mathbb{Z}_2$ or the action equivalent to the $T_2-$action in the matrix representation. 

Suppose that $\overline{S'}$ is isomorphic to $\mathbb{Z}_2.$ If $C_S(A)\cap S'$ is equal to $A$ we get an element $y$ with orbit $\{y,\ w\}.$ Moreover there exists $x$ in $S'$ such that its projection to $\overline{S'}$ generates the group and therefore we get that $x, y$ and $z$ generate $S',$ where $z$ is an element in $A$ which together with $y$ and $w$ generate $A.$

Similarly when $C_S(A)\cap S'$ is equal to $C_S(A)$ we can choose an element $y$ in $A$ with orbit $\{y,\ w\}$ and an involution $z$ such that $y,\ w,\ z, \ f $ generate $C_S(A),$ where $f$ is an element of order 4 in $C_S(A).$ In addition we can take $x$ in $S'$ such that coset $xC_S(A)$ generates $\overline{S'}$ and thus $x,\ y,\ z, \ f $ generate $S'.$ 

The last possible case is when $\overline{S'}$ is isomorphic to $\mathbb{Z}_2\times\mathbb{Z}_2$ and the action is the same as $T_2$, that is we have three non-trivial orbits of two elements. 

If $C_S(A)\cap S'=A$, we consider $y$ and $z$ two elements in $A$ contained in distinct orbits of 2 elements and $r$ and $x$ such that $rA$ and $xA$ generate  $\overline{S'}.$ The elements $r,\ x,\ y,\ z$ generate $S'.$ 

Finally, we suppose that $C_S(A)\cap S'=C_S(A)\cong \mathbb{Z}_2\times\mathbb{Z}_2\times\mathbb{Z}_4.$ We denote by $f$ an element of order 4 in $C_S(A)$. An element of order 4 can be either a $K_1-$rotation or an exchange element of $Fix(h)$. Since $A$ contains the $K_1-$reflection $t$ and $t$ commutes with $f$, then $f$ is an exchange element. The involution $f^2$ is central in $S'$ since it is the unique involution that is a square of an element in $C_S(A)$. We can suppose that $f^2$ acts freely  otherwise we can conclude by Lemma \ref{lemma: 2}. The only remaining possibility is that $f^2$ is a $K_1-$ rotation. 

We consider $C_{S'}(h)$. The action of each involution in $\overline{S'}$ have exactly two orbits of 2 elements.  Therefore $C_{S'}(h)$ is a subgroup of index 2 in $S'$ (if $h$ is central we are done by Lemma \ref{lemma: 2}).   Let $H_1,\ H_2,\ H_3$ be subgroups of $C_S'(h),$ where $H_1$  are the trivial $K_1$-rotations,  $H_2$ are the $K_1-$rotations and $H_3$ are all elements fixing $K_1$ set-wise. Then we have  
\begin{align*} 
H_1\vartriangleleft H_2 \vartriangleleft H_3 \vartriangleleft H_4=C_{S'}(h)\vartriangleleft H_5=S', 
\end{align*} 
with $H_i/H_{i-1}$ cyclic for $i=2, 3,  4,  5$ and for $i=3, 4, 5$ it has order 2. Therefore, we get a set of  five generators for $S'.$ Moreover if the element conjugate to $h$ is a $K_1$-reflection or an exchange element then we have easily a set of four generators. Thus, we can suppose that the conjugate element to $h$ acts as rotation on $K_1$. The only possibility is that $h$ is conjugate to $f^2h.$
We obtain that  $M/\langle f^2h\rangle$ is a $\mathbb{Z}_2$-homology 3-sphere (conjugate elements have diffeomorphic quotients).  The involution $h$ projects to an involution $\overline{h}$ of  $M/\langle f^2h\rangle$. Since $f^2h$ is a  $K_1-$rotation we obtain that  $\overline{h}$ has a non-connected fixed-point set and this is impossible. \newline
\\
\textbf{Case 1.1.2} Suppose $C_S(A)$ is an elementary abelian group of rank 4. 
\smallskip

By what said so far there exist  $f,\ t,\ s$ elements of $C_S(A),$ where $f$ is a $K_1-$rotation, $t$ is a $K_1-$reflection and $s$ is an exchange element of $Fix(h).$ \\ 
We have that $C_S(A)$ is generated by $\{h,\ f,\ t,\ s\}$ and in particular  
\begin{align*} 
C_S(A)=\{1,\ h,\ f,\ fh,\ t,\ th,\ tf,\ tfh,\ s,\ sh,\ sf,\ sfh,\ st, \ sth,\ stf,\ stfh  \}, 
\end{align*} 
which is a total of four $K_1$-rotations (including the identity), four $K_1$-reflections and eight exchange elements.

The reflections have nonempty fixed-point set and by Remark \ref{Remark: 2} we may suppose that the four reflections of $C_S(A)$ have fixed-point set with two components; more generally, we can suppose that the fixed-point set of the non-trivial elements set of $C_S(A)$ either is empty or  has two components. 
Now the two elements $f$ and $fh$, which act as non-trivial $K_1$-rotations, can have either both empty or one empty and one nonempty fixed-point set. In fact, the case where both have nonempty fixed-point set can be excluded by the fact that they are $K_1$-rotations. If they both have nonempty fixed-point set then $Fix(f)$ and $Fix(fh)$ have nonempty intersection, since otherwise we would get a separation of the fixed-point set of the involution acting on  $M/\langle h\rangle$ induced by  $f$. Therefore $Fix(f)$ and $Fix(fh)$ intersect $Fix(h)$; given that they have two components, it follows  that they necessarily meet $K_1,$ which is a contradiction.

\smallskip

We consider first the case where the fixed-point set of both $f$ and $fh$ is empty. If $g$ is an element of $C_S(h)$ we denote by $\overline{g}$ the action induced by $g$ on $M/\langle h\rangle.$ We may consider the following group $\{\overline{1}, \ 
\overline{f},\ \overline{s},\ \overline{fs} \}$ acting on $M/\langle h\rangle,$ which is isomorphic to $\mathbb{Z}_2\times\mathbb{Z}_2.$  We know by Lemma \ref{lemma MZ} since $Fix(\overline{f})$ is empty that $Fix(\overline{s})$ and $Fix(\overline{sf})$ are both nonempty with empty intersection. This implies that one among $Fix(s)$ and $Fix(sh)$ is nonempty. We can exclude the case where both are nonempty, given the fact that they are exchange elements: since $Fix(\overline{s})$ is connected, $Fix(s)$ and $Fix(sh)$ should  intersect $Fix(h)$ to avoid  a separation of $Fix(\overline{s}).$ The same argument can be applied to the pair $sf$ and $sfh$ and thus one of the two has nonempty fixed-point set.

We can consider the group $\{ \overline{1},\ \overline{f},\ \overline{st}, \ \overline{stf}\}$ acting on $M/\langle h\rangle$ and by the  same reasoning we get that one among $st$ and $sth$ and one among $stf$ and $stfh$ have nonempty fixed-point set.

We have obtained nine elements with nonempty fixed-point set:  h, four reflections and four exchange elements.  An element with nonempty fixed-point set cannot be conjugate to an element with empty fixed-point set. More generally, two conjugate elements in $S$ have fixed-point set with the same number components.  We get therefore that $S$ acts on this set of nine elements by conjugation and  we have one element in the center of $S$; finally, by Lemma \ref{lemma: 2} we can conclude.

\smallskip

Suppose that  one among $f$ and $fh$ has nonempty fixed-point set. Consider the two quadruplets of exchange elements  
\begin{align*} 
 E_1=\{s,\ sh,\ sf,\ sfh\}, \ \  E_2=\{st,\ sth,\ stf,\ stfh \} 
\end{align*} 

Consider the action of $\{\overline{1},\ \overline{s},\ \overline{f},\ \overline{sf} \}$ on $M/\langle h\rangle.$ By assumption  $Fix(\overline{f})$  is connected and using Lemma \ref{lemma MZ} we get that at least one among $\overline{s}$ and $\overline{sf}$ has nonempty fixed-point set. In the same way by considering the action of $\{\overline{1},\ \overline{st},\ \overline{f},\ \overline{fst} \}$ on $M/\langle h\rangle$ we obtain that at least one among $\overline{st}$ and $\overline{stf}$ has nonempty fixed-point set. 

If $\overline{s}$ has nonempty fixed-point set, by the same argument used above, we can obtain that exactly one between $s$ and $sh$ has nonempty fixed-point set. The same holds for the other exchange elements. Therefore, we get that there are one or two elements in $E_1$ with nonempty fixed-point set and the same for the set $E_2.$\\
\\ 
\textbf{Claim:} We cannot have 2 elements with nonempty fixed-point set in both $E_1$ and $E_2.$\\ 
\\ 
\textbf{Proof of the claim:} Suppose that there are 2 elements in $E_1$ and 2 elements in $E_2$ with nonempty fixed-point set. Now since the elements in $E_1$ and $E_2$ act as exchange element on $Fix(h)$ we have that for any $y$ in $E_1\cup E_2$ only one among $y$ and $yh$ can have nonempty fixed-point set. 

Moreover if we consider the actions of $\{\overline{1}, \ \overline{s}, \overline{f},\ \overline{sf}\}$ and $\{\overline{1}, \ \overline{st}, \overline{f},\ \overline{stf}\}$ on $M/\langle h\rangle$ we get that both groups have three involutions with nonempty fixed-point set meeting in exactly two points. Therefore, there exists $P$ in $M$ whose projection to $M/\langle h\rangle$ is fixed by $\overline{s}, \overline{f}$ and $\overline{sf}.$ 

Assume that $f$ has nonempty fixed-point set. The point  $P$ is fixed by $f$ and by either $s, \ sf$ or $sh,\ shf.$ Moreover this implies that either $s$ and $sf$ or $sh$ and $sfh$ act as reflection on a component of $Fix(f)$. Since   $C_S(A)$ is an elementary group of rank 4, it contains exchange elements of $Fix(h).$  This implies that an  involution of $C_S(A)$ acting as a reflection on a component of $Fix(f)$, acts as a reflection also on the other component. In the same way we can see that either $st,\ stf$ or $sth,\ stfh$ act as a reflection on $Fix(f)$.
 
Further  we consider the action of $\{\overline{1}, \ \overline{f}, \overline{t},\ \overline{tf}\}$ on $M/\langle h\rangle$; here each involution does not act freely on $M/\langle h\rangle$ and, again, we have two global fixed-points. It follows that either $t,\ tf$ or $th,\ tfh$ fix $Fix(f)$ component-wise and act on both of its components as a reflection. Finally, we obtain at least  six involutions acting as reflections on a single component of  $Fix(f)$ and this is impossible.  The same argument holds if $fh$ has nonempty fixed-point set. This concludes the proof of the claim.   

\medskip

We can therefore suppose that exactly one among $\overline{s},$  $\overline{sf}$ and exactly one among $\overline{st},$  $\overline{stf}$ have empty fixed-point set, since in the case where one in $E_1$ and two in $E_2$ or vice versa have nonempty fixed-point set leads again to an odd number of elements with nonempty fixed-point set. Thus, we can choose one exchange element out of each quadruple $E_1$ and $E_2$ inducing an involution acting freely on $M/\langle h\rangle$ and  consider the group generated by these two free involutions acting on $M/\langle h\rangle$, this is impossible by Lemma \ref{lemma MZ}.

\smallskip  
This concludes the case where $A$ contains a reflection $t$.\\ 
\\ 
\textbf{Case 1.2} Suppose now that $A$ does not contain a $K_1$-reflection. Hence $A$ contains a rotation $f$ and exchange element $s.$ We have that $A/\langle h\rangle$  is isomorphic to $\mathbb{Z}_2\times\mathbb{Z}_2$  and by Lemma~\ref{lemma MZ}   we may distinguish two cases. The first case is where the fixed-point sets of two involutions have nonempty fixed-point set with empty intersection, the second case where all three involutions have nonempty connected fixed-point set and all three intersect in exactly two points.

\smallskip

Suppose first we have two involutions with nonempty fixed-point set. We can see in similar fashion as already done that in this case $A$ has exactly three elements which have nonempty fixed-point set. $S$ acts by conjugation on these three elements and thus one of these elements lies in the center of $S$ and we can conclude by using Lemma \ref{lemma: 2}.

\smallskip

Suppose now that $A/\langle h\rangle$ has 3 involutions with nonempty fixed-point set. This implies that there exist $t_1,\ t_2,\ t_3$ in $A$ with nonempty fixed-point set, while $t_1h,\ t_2h,\ t_3h$ have empty fixed-point set. $S$ acts on the set $X=\{t_1,\ t_2,\ t_3,\ h\}. $ We claim that $h$ cannot be conjugate to any of $t_1,\ t_2,\ t_3.$ and thus $h$ lies in the center and, if the claim is true, we can conclude again by Lemma \ref{lemma: 2}.

To prove the claim, we may notice that by assumption  $\overline{t_1},\ \overline{t_2},\ \overline{t_3}$ have fixed-point set intersecting with each other in two points. 
Suppose by contradiction that $h$ is conjugate to any other element of $X, $ i.e. $t_i^u=h,$ for some $i\in\{1,2,3\}$ and some $u$ in $S.$ By assumption some $t_k$, for $k$ different from $i$ acts as reflection on a component of $Fix(t_i).$ Thus there exists $P$ in $Fix(t_i)$ fixed by $t_k. $ Moreover it is immediate to see that $u(Fix(h))=Fix(t_i).$ Thus there exists $\Tilde{P}$ in $Fix(h)$ such that $u(\Tilde{P})=P$ which implies $u^{-1}t_ku(\Tilde{P})=\Tilde{P}, $ thus we have found an element of $A$ acting as reflection on the components of $Fix(h)$ contrary to the assumption. This proves the claim and hence concludes Case 1.\\ 
\\ 
\textbf{Case 2} The subgroup $A$ does not contain any hyperelliptic involution.
\smallskip

By hypothesis we know that $S$ contains a hyperelliptic involution with fixed-point set with two components; by the assumption of Case 2 we have that it is not contained in $A$.
Consider therefore $B$ the group generated by $A$ and by the hyperelliptic involution; it has order 16 and, by looking at the possible groups containing $\mathbb{Z}_2^3$ as a subgroup, we are left with  four groups. We can conclude that $B$ is isomorphic to $\mathbb{Z}_2^4$ or to $\mathbb{D}_8\times\mathbb{Z}_2,$ since the other two possibilities have only  elements of order 4 outside the subgroup isomorphic to $\mathbb{Z}_2^3.$\\  
\\ 
\textbf{Case 2.1} Suppose that $B$ is isomorphic to $\mathbb{D}_8\times\mathbb{Z}_2.$ 

\smallskip

we have the following presentation  
\begin{align*} 
B=\langle z, f, t| \ z^2=t^2=f^4=1, zf=fz, zt=tz, tft=f^{-1}\rangle. 
\end{align*}  
We have that $Z(B)=\{1,\ z, \ f^2, f^2z \}.$ \\ 
The following is the set of all elements of order 2 of which at least one is hyperelliptic: 
\begin{align*} 
I(B)=\{f^2,\ t, \ z,\ tf,\ tf^2,\ tf^3,\ zf^2,\ zt,\ ztf,\ ztf^2,\ ztf^3\} \end{align*} 
Moreover $B$ has two subgroups isomorphic to $\mathbb{Z}_2^3$ 
\begin{align*} 
A_1 &=\{1, \ f^2,\ z, \ f^2z,\ t, \ tf^2, \ tz,\ tf^2z \}\\ 
A_2 &=\{1, \ f^2,\ z, \ f^2z,\ tf, \ tf^3, \ tfz,\ tf^3z \}, 
\end{align*} 
and we can see that any element of $I(B)$ is contained in either $A_1$ or $A_2.$

We can observe that $f^2,\ z,\ f^2z$ are necessarily contained in $A$ and thus are not hyperelliptic and since they are in the center of $B$ we get by Lemma \ref{lemma: 1} and Remark \ref{Remark: 1} that they act freely on $M$ or they have fixed-point set with two components.

The conjugacy classes of $B$ which we are interested in are the following:
\begin{align*} 
 \{t, \ tf^2\}, \ \{tz, \ tf^2z\}, \ \{tf, \ tf^3\}, \ \{tfz, \ tf^3z\}. 
\end{align*}

The group $B$ contains a hyperelliptic involution with fixed-point set with two components; we  can suppose w.l.o.g. that this hyperelliptic involution  is  $t$. Therefore, the conjugate involution  $tf^2$ is also hyperelliptic with a 2-component fixed-point set and  $A_2=A.$

In addition, in the following cases we consider the diffeomorphism induced by an element of $S$ on a fixed  quotient of $M$ (that can be vary in different situations); we use the following notation:  if $g$ is an element normalizing the subgroup $C$ of $S$, we denote by $\overline{g}$ the diffeomorphism induced by $g$ on $M/C.$\\ 
\\ 
\textbf{Case 2.1.1} Suppose that $Fix(t)$ and $Fix(tf^2)$ do not intersect.

\smallskip

This implies that also $Fix(tf^2)$ and $Fix(f^2)$ do not intersect and moreover only one among $tf^2$ and $f^2$ has nonempty fixed-point set, since otherwise we get a separation of $Fix(\overline{f^2})=Fix(\overline{tf^2})$ in $M/\langle t\rangle.$ Hence $f^2$ acts freely, since by assumption $tf^2$ has nonempty fixed-point set.

We can observe that $tf^2$ is an exchange element of $Fix(t).$ Indeed $tf^2$ fixes $Fix(t)$ set-wise but if $tf^2$ fixes $Fix(t)$ component-wise then we can find a separation of $Fix(\overline{t})$ in $M/\langle tf^2\rangle.$ Vice versa $t$ is an exchange element for $Fix(tf^2)$ and therefore it follows straight forwardly that $f^2$ is an exchange element for both $Fix(t)$ and $Fix(tf^2).$
 

Suppose first that in addition to $f^2$ also $z$ and $f^2z$ act freely on M and consider the action of $A_1/\langle t\rangle=\{\overline{1},\ \overline{f^2},\ \overline{z},\ \overline{f^2z} \}$, which is isomorphic to $\mathbb{Z}_2\times\mathbb{Z}_2$, on $M/\langle t\rangle.$ By Lemma \ref{lemma MZ} at least two among $\overline{f^2}=\overline{tf^2},\ \overline{z}=\overline{tz},\ \overline{f^2z}=\overline{tf^2z}$ have nonempty fixed-point set. Moreover since $tz$ and $tf^2z$ are conjugate, we get that $\overline{tf^2},\ \overline{tz}$ and $\overline{tf^2z}$ have nonempty fixed-point set. Hence $tf^2, \ tz$ and $tf^2z$ have nonempty fixed-point set. Moreover again by Lemma \ref{lemma MZ} there exists $P$ whose projection to $M/\langle t\rangle$ is fixed by $\overline{tf^2},$ $\overline{tz}$ and $\overline{tf^2z}$ and since we assume that $f^2,$ $z$ and $f^2z$ act freely we get that $P$ is fixed by $tf^2,$ $tz$ and $tf^2z.$ This last conclusion leads to a contradiction, since a point $P$ fixed by $tf^2$ and $tz$ is also fixed by $f^2z$ which by assumption acts freely. 
We can conclude that at least one among $f^2z$ and $z$ does not act freely.

Consider the case where only one among $z$ and $f^2z$ has nonempty fixed-point set. We  recall that by assumption $f^2$ has empty fixed-point set. Moreover $tf$ is conjugate to $tf^3$ and $tfz$ to $tf^3z$ in $B,$ thus we get that the number of elements with nonempty fixed-point set in $A$ is odd. Therefore, we are done since we get that an element among $z$ and $f^2z$ having nonempty fixed-point set is in the center of $S$ and thus we can conclude by Lemma \ref{lemma: 2}.

We can therefore suppose that $z$ and $zf^2$ have nonempty fixed-point set with two components. Consider the  action of $\{\overline{1},\ \overline{f^2},\ \overline{z},  \ \overline{f^2z}\}$ on $M/\langle tf^2\rangle.$ Now since $\overline{f^2}=\overline{t},$ and  $z,\ f^2z$ have nonempty fixed-point set we have that $\overline{f^2},\ \overline{z}, \ \overline{f^2z}$ have nonempty fixed-point set intersecting in two points. Moreover since $f^2$ acts freely there exists a point $P$ in $M$ fixed by $t$ whose projection is fixed also by $\overline{z}$ and $\overline{f^2z}.$ Hence $P$ is fixed by $z$ or $tf^2z$ and by $f^2z$ or $tz.$ The case where both representatives of $\overline{z}$ and $\overline{tf^2z}$ fix $P$ can be excluded since this would contradict the fact $tf^2$ is an exchange element for $Fix(t).$ If we suppose that $z$ fixes $P$ then we get that also $tz$ fixes $P$ and therefore $z$ and $tz$ act as reflections on a  component of $Fix(t);$ this implies that $z$ and $tz$ fix  $Fix(t)$ component-wise. Analogously, if we suppose that $tf^2z$ fixes $P$ so does $f^2z$ and both act as reflections on a  component of $Fix(t).$

Assume that $z$ and $tz$ act as  reflections on a component of $Fix(t).$ Thus since $f^2$ is an exchange element on $Fix(t)$ also $f^2z$ and $tf^2z$ are. In addition, we have that, since $tz$ and $tf^2z$ are conjugate, $tf^2z$ has nonempty fixed-point set. We know by assumption that $f^2z$ has nonempty fixed-point set therefore there exists $\Tilde{P}$ in $M$ fixed by $f^2z$ and $tf^2z$ since otherwise we would get a separation of the fixed-point set of the involution induced by $f^2z$ on $M/\langle t\rangle.$ Thus $\Tilde{P}$ is fixed also by $t$ and therefore also $f^2z$ and $tf^2z$ act as reflections on a component of $Fix(t)$ fixing component-wise $Fix(t)$, which is impossible.
 
The same argument can be applied if we suppose at first $f^2z$ and $tf^2z$ act as reflections instead of $z$ and $tz.$\\ 
\\ 
\textbf{Case 2.1.2 } Suppose now that $Fix(t)$ and $Fix(tf^2)$ have nonempty intersection and therefore $f^2$ has nonempty fixed-point set. 

\smallskip

Since the nonempty fixed-point set of $f^2$ has two components, if $f^2$ is in the center of $S$ we are done by Lemma \ref{lemma: 2}. Therefore, we suppose that $f^2$ is not in the center of $S$ and  is conjugate to another element of $A.$  We recall that $f^2,\ z,\ f^2z$ are in the center of $B$ and the conjugacy classes of $A$ in $B$ with more than one element are $\{tf,\ tf^3\}$ and $\{tfz,\ tf^3z\}$. This tells us that one among $f^2,\ z,\ f^2z$ must be in the center of $S$ and the remaining two are conjugate; indeed $A$ can have either three conjugacy classes of size 2 in $S$ or one conjugacy class of size 4 and one of size 2.



We claim now that $C_S(f^2)$ has index two in $S.$  By the previous discussion we have that either $f^2$ is conjugate to $z$ or $f^2$ is conjugate to $f^2z.$ Suppose first that $f^2$ is conjugate to $z$. Therefore, there exists $x$ in $S$ such that $(f^2)^x=z.$ Moreover we have that $(tf^3)^x=(tf)^x(f^2)^x=(tf)^xz.$ Now since $f^2z$ is in the center $(tf)^x$ cannot be equal to $f^2$ and thus we get that it must be equal to one among $tf,\ tf^3,\ tfz, \ tf^3z$. Suppose for example $(tf)^x=tf^3, $ then $(tf^3)^x=(tf)^x(f^2)^x=tf^3z$ and therefore $\{tf,\ tf^3,\ tfz, \ tf^3z\}$ is a conjugacy class in S and the conjugacy class of $f^2$ in $S$ contains two elements. Analogously we can conclude for the other cases. Suppose secondly that $f^2$ and $f^2z$ are conjugate and therefore $(tf^3)^x=(tf)^x(f^2)^x=(tf)^xf^2z.$ From which follows again that $\{tf,\ tf^3,\ tfz, \ tf^3z\}$ is a conjugacy class. This finishes the proof of the claim.

Let $L_1$ and $L_2$ be the components of $Fix(f^2).$  Let $H_1,\ H_2,\ H_3$ be subgroups of $C_S(f^2),$ where $H_1$  are the trivial $L_1$-rotations,  $H_2$ are the $L_1-$rotations and $H_3$ are all elements fixing $L_1$ set-wise. Then we have  
\begin{align*} 
H_1\vartriangleleft H_2 \vartriangleleft H_3 \vartriangleleft H_4=C_S(f^2)\vartriangleleft H_5=S, 
\end{align*} 
with $H_i/H_{i-1}$ cyclic for $i=2, 3,  4,  5$ and for $i=3, 4, 5$ it has order 2. Therefore, we get a set of five generators for $S.$ Moreover if the element conjugate to $f^2$ is a $L_1$-reflection or an exchange element then we have easily a set of four generators. Moreover, if $S'\subset S$ we can consider the subnormal series induced by $\{H_i\}$ to obtain that $S'$ is generated by a set of  four generators.  Thus, we can suppose that the conjugate element to $f^2$ acts as rotation on $L_1$. We note that we can repeat the same argument also for $L_2$. 

We may recall that $f^2$ is in the center of $B$, therefore  our hyperelliptic involution $t$ is in $C_S(f^2).$ We denote by  $\overline{f^2}$ the involution acting on $M/\langle t\rangle$ induced by $f^2.$  Since $M/\langle t\rangle$ is a 3-sphere, we have that  $\overline{f^2}$ has connected fixed-point set. By assumption $Fix(t)$ and $Fix(f^2)$ have nonempty intersection which implies that $t$ acts as a reflection on both components of $Fix(f^2)$, otherwise we have a separation of the fixed-point set of $\overline{f^2}$. 

We consider a successive quotient $N=M/\langle t, f^2\rangle$. Since we can obtain $M/\langle t, f^2\rangle$ as  $(M/\langle t \rangle)/\langle \overline{f^2}\rangle$, by the positive solution of Smith Conjecture we obtain that  $N$ is diffeomorphic to $S^3.$ We focus on  the projection of the fixed-point set of $t$, $f^2t$ and $f^2$ on $N$, which  we denote by $\Gamma$. The combinatorial structure of $\Gamma$ is represented by one of the trivalent graphs in Figure \ref{fig:trigraph}; the projection of the fixed-point sets of $t,$ $f^2,$  and $tf^2$ are respectively black, red and blue. 

\begin{figure}[h] 

   \centering 

   \includegraphics[width=0.8\textwidth]{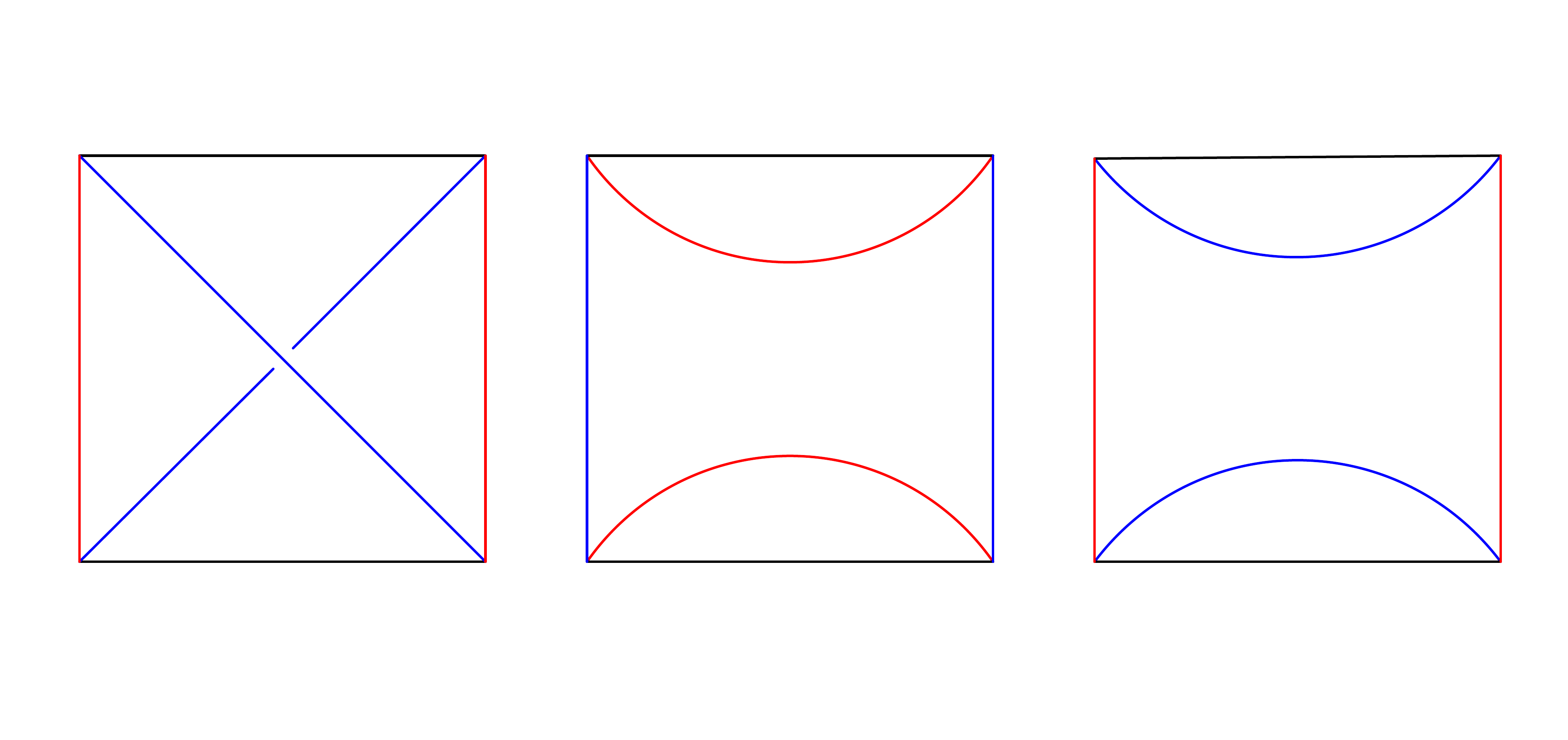} 

   \caption{The three possible combinatorial structures of $\Gamma$; the projection of the fixed-point sets of $t,$ $f^2,$  and $tf^2$ are respectively black, red and blue} 

    \label{fig:trigraph} 
\end{figure}

The manifold $M/\langle f^2 \rangle$ can be obtained as the cyclic branched cover along the union of the black and blue curves. If the union is a loop then $M/\langle f^2 \rangle$ is a $\mathbb{Z}_2-$homology sphere and we can conclude as in Case 1. There is only one remaining graph where the union gives two loops that is  the third one in Figure \ref{fig:trigraph}. We suppose that we have this graph. We denote by $r$ the rotation conjugate to $f^2$, which can be $z$ or $zf^2$.  In any case, $r$  is in the center of $B$ and commutes with all the elements of $\langle t, f^2\rangle$.  It induces an involution $\overline{r}$ acting  on $N$ that leaves invariant set-wise $\Gamma$. Moreover, since $r$ centralizes $\langle t, f^2\rangle$, the involution $\overline{r}$ fixes set-wise each set of curves of a fixed color. The fact that $z$ acts as a rotation on both components of $Fix(f^2)$ implies that each red curve is left invariant set-wise by  $\overline{r}$. This implies that $\overline{r}$ has to exchange the two black curves of $\Gamma$ and we can conclude that $r$ is an exchanging element of $Fix(t)$.

We consider now the quotient $N'=M/ \langle t, r\rangle$. It can be seen as the quotient  of $M/ \langle t \rangle$ by the action of  the involution induced by $r$. This implies that  $N'$ is a 3-sphere and the projections of the union of the fixed-point sets of the involutions in  $\langle t, r\rangle$ form a link with two components. In fact, the involution $tr$ acts freely (otherwise we get a separation of the fixed-point set of the involution induced by $r$) and the fixed-point set of $t$ (resp. $r$) projects  to one of the component of the link. We can obtain  $M/ \langle r\rangle$ as the 2-fold cyclic branched cover of   $N'$ along the projection of the fixed-point set of $t$. Finally, we obtain that $M/ \langle r\rangle$ is a  $\mathbb{Z}_2-$homology sphere and, since $r$ is in $A$, we can again conclude as in Case 1. This concludes the case $B$ isomorphic to $\mathbb{D}_8\times \mathbb{Z}_2$. 

Indeed, we have just proved a  general claim that will be used again in the following of the proof:\\
\\ 
\textbf{Claim:}  If $A$ contains an involution not acting freely, commuting with a hyperelliptic involution and acting as an exchange element on the fixed-point set of the hyperellptic involution, we are in Case 1.
\newline 
\\ 
\textbf{Case 2.2} Assume now that $B$ is isomorphic to $\mathbb{Z}_2^4.$
\smallskip

We denote by $h$ the hyperelliptic involution in $B$. We are looking for a subgroup of $B$ which has order 8 and does not contain $h.$ We denote by $L$ one of the components of $Fix(h).$ We know by what done so far that $B$ is generated by $h,$ a non-trivial $L$-rotation $f$, an $L$-reflection $t$ and an exchange element $s.$\\ 
In particular 
\begin{align*} 
B=\{& 1,\ h, \ f,\ fh,\ t,\ th,\ tfh,\ tf, \ s,\ sfh,\ sh,\ sf,\ st,\ stfh,\ sth,\ stf \},
\end{align*} 
and we might consider the following sets: 

\begin{align*} 
&\{  h, \ f,\ fh\}, \  \{t,\ th,\ tfh,\ tf\},\  \\ 
&\{s,\ sfh,\ sh,\ sf,\ st,\ stfh,\ sth,\ stf \},
\end{align*} 
which are the $L$-rotations excluded the identity, the $L$-reflections and the exchange elements of $Fix(h)$, respectively. From now on we refer to these elements simply as rotations, reflections and exchange elements, respectively.

We can observe that any subgroup of $B$ of order $8$ contains a rotation since the product of two reflections gives a rotation and the product of two exchange elements gives either a reflection or a rotation.

Moreover two rotations generate all rotations and hence $h.$ This implies that any subgroup of $B$ of rank 3 not containing $h$ must contain an exchange element since if we have a subgroup generated by one rotation and two reflections then the two reflections give another rotation which is either equal to the given rotation and thus the subgroup has order four or the subgroup has two distinguished rotations and contains $h.$

We get therefore that any subgroup of $B$ of order 8 not containing $h$ is generated by a rotation, a reflection and an exchange element. In addition, this tells us that each such subgroup contains one non-trivial rotation, two reflections and four exchange elements.

Since each element in $B$ commutes with $h$, we can suppose that the involutions of $B$ not acting freely have a fixed-point set with two components.  By Lemma \ref{lemma: 2} we can suppose that none of these elements is central in $S$. Therefore, we can suppose that the elements with nonempty fixed-point set in $A$ are either two, four, or six, since in the case where we have an odd number we have a  central involution not acting freely. \\ \\
\textbf{Case 2.2.1} Suppose first that $A$ contains at least four elements not acting freely.
\smallskip

We have that at least one of these elements with nonempty fixed-point set is an exchange element. By using the claim stated in the end of  Case 2.2.1 we can conclude by Case 1.  \\
\\
\textbf{Case 2.2.2} Suppose secondly that $A$ contains two elements with nonempty fixed-point set. 
\smallskip

This implies that there  exists an element $g$ with nonempty fixed-point set such that $C_S(g)$ has index 2 in S.  

Now let $\{L_1,\ L_2\}$ be the components of the fixed-point set of $g$ and suppose that $A$ contains either a $L_1-$reflection or a $L_2-$reflection. Moreover in this case it follows that $C_S(A)$ is isomorphic to either $\mathbb{Z}_2^4$ or to $\mathbb{Z}_2\times \mathbb{Z}_2\times\mathbb{Z}_4$. Since the involution $h$ is not in $A$ but is in $C_S(A)$, we can conclude that   $C_S(A)$ is elementary abelian of rank 4 containing $h$ and we can repeat the same argument of the Case 1.1.2.

We can therefore suppose that $A$  contains neither a $L_1-$reflection nor a $L_2-$reflection and therefore $g'$ the element conjugate to $g$ acts as rotation on $Fix(g)$ since as done before we can construct a subnormal series from which if $g'$ is an exchange element we can conclude that has $S$ sectional rank of at most 4. We may observe that $h$ acts as exchange element or as reflection on the components of $Fix(g).$ In addition since in $A$ there is a single element acting as rotation on the components of $Fix(h)$ we may suppose that $g$ acts as exchange element $Fix(h)$ (otherwise we can consider $g'$ instead of $g$). Again by the claim stated at the end of Case 2.2.1 we can conclude. 

\rightline{\qedsymbol{}}  

\section{Simple groups}\label{simple groups}
Let $G$ be a simple group acting on $M$ containing an hyperelliptic involution $h$ such that the fixed-point set has two components. 
\\
The sectional 2-rank of $G$ is less or equal to four: we can use the Gorenstein-Harada Theorem, see   \cite[Theorem 8.12 and the note at page 517]{suzuki1986group}.
We have that $G$ is isomorphic to one the groups in the following list (where $q$ is an odd prime number):
\begin{itemize}
    \item $PSL(2,q)$ with  $q>3$, $PSL(3,q),$ $PSL(4,q)$ with $q\neq 1$ mod 8, $PSL(5,q)$  with $q=-1$ mod 8;
    \item $PSU(3,q),$ $PSU(4,q)$ with $q\neq -1$ mod 8, $PSU(5,q)$  with $q=1$ mod 8;
    \item $PSp(4,q);$ $G_2(q)$ with $q>3$; $D_4^3(q)$; $^2G_2(3^{2m+1})$;
    \item groups of Lie type of even characteristic: $PSL(2,8),$   $PSL(2,16),$ $PSL(3,4),$ $PSU(3,4),$ and $Sz(8)$;
    \item the alternating groups: $\mathbb{A}_n$, for $7\leq n\leq 11$;
    \item the sporadic groups: $M_{11},$ $M_{12},$ $M_{22},$ $M_{23},$ $J_1,$ $J_2,$ $J_3,$ $L_y,$ $M^c$;
\end{itemize}
We identify the groups in this list which have at least two conjugacy classes of involutions; these groups are:
\begin{align*}
  &PSL(4,q) \ for \ some \ q,\  PSL(5,q),\ PSU(4,q) \ for\ some\ q,\\ 
 &PSU(5,q), \ PS_p(4,q),\ \mathbb{A}_n\ for\ 8\leq n\leq 11; \ M_{12};\  J_2.  
\end{align*}
The number of conjugacy classes can be found in \cite[Table 4.5.1.]{GLS3} or \cite[Chapter 6.5]{suzuki1986group} for groups of Lie type of odd characteristic,  and in  \cite{Atlas} for  groups of Lie type of even characteristic and  sporadic groups. For the groups with at least two conjugacy classes of involutions, we have that  each involution has a  centralizer either non solvable or containing  a non-normal Sylow 3-subgroup. Hence, they cannot occur by Proposition \ref{centralizer}. To obtain the centralizers, see again \cite[Table 4.5.1.]{GLS3} or \cite[Chapter 6.5]{suzuki1986group} for groups of Lie type; for Alternating group the statement follows from direct computations;  the centralizers of the involutions of the sporadic groups are given in \cite[p. 253 and 268]{GLS3}.\\
Thus, we can assume that $G$ has a single conjugacy class of involutions. Moreover all involutions in $G$ are hyperelliptic. 
\begin{proposition}
The $2-$rank of $G$ is at most two.
\end{proposition}
\begin{proof}
    Suppose that $H$ is a subgroup of $G$ which is elementary abelian of rank three. Let $h $ be an element of $H$. All elements $f$ in $H$ different from $h$ act as a reflection on both components of $Fix(h),$ since otherwise we obtain a separation of projection of $Fix(f)$ to $M/\langle h\rangle,$ which is connected. However, this is impossible, since by assumption $H$ has rank three. 
\end{proof}
This result allows us to further restrict the possible groups. We can use \cite[Theorem 6.8.16]{suzuki1986group} and we obtain that $G$ is isomorphic to one of the following groups: 
\begin{align*}
    PSL(2,q),\ PSL(3,q),\ PSU(3,q), \ PSU(3,4),\ \mathbb{A}_7, \ M_{11}, 
\end{align*}
where $q$ is an odd prime power.\\
The groups  $PSL(3,q)$ and  $PSU(3,q)$ can be excluded  again by Proposition \ref{centralizer} (the centralizers of the involutions are described again in \cite[Table 4.5.1.]{GLS3} and \cite[Chapter 6.5]{suzuki1986group}. The group  $PSU(3,4)$ can be excluded since the Sylow 5-subgroup of the centralizer of an involution is not normal (by an explicit computation using Sage). The centralizer of an involution of $M_{11}$ is isomorphic $GL(2,3)$ (see  \cite[p.262]{GLS3}) whose Sylow 3-subgroups are not normal.\\
Moreover, we have the following result:
\begin{proposition}
The alternating group $\mathbb{A}_7$ cannot act on $M$ and contain a hyperelliptic involution with two components. 
\end{proposition}
\begin{proof}
    Suppose $\mathbb{A}_7$ acts on $M$ and contains a hyperelliptic involution with two components. Consider $t=(1,2)(3,4)$, $s=(1,3)(2,4)$, $st=(1,4)(2,3)$ and $f=(5,6,7).$ We have that $t,$ $s$ and $st$ have fixed-point set with two components. Moreover $t,$ $s$ and $f$ generate a subgroup of $\mathbb{A}_7$ isomorphic to $\mathbb{Z}_2\times\mathbb{Z}_2\times\mathbb{Z}_3.$ Thus $f$ fixes the fixed-point set of $t$, $s$ and $st$ set-wise. Since $f$ is an element of order three we get that $f$ acts necessarily as rotation on each component of $t$, $s$ and $st.$ In addition we know that the set $Fix(t)\cap Fix(s)\cap Fix(st)$ contains exactly four points. Suppose that $\{K_1,K_2\}$ are the connected components of $Fix(t)$, $\{L_1,L_2\}$ the connected components of $Fix(s)$ and $\{N_1,N_2\}$ the connected components of $Fix(st).$ If $P$ is in $K_1\cap L_1\cap N_1$ then so is $f(P)$, since $f$ acts as rotation on $K_1, L_1$ and $N_1.$ Similarly given a point $Q$ in $K_2\cap L_2\cap N_2$ we have that $f(Q)$ is in $K_2\cap L_2\cap N_2$. What we have obtained is 
    \begin{align*}
        &K_1\cap L_1\cap N_1=\{P, f(P)\} \\
        & K_2\cap L_2\cap N_2=\{Q, f(Q)\}.
    \end{align*}
    However it has been shown in the proof of Lemma \ref{lemma: 1} that this is impossible.
\end{proof}
Hence, we have obtained that the only possible simple groups are the linear fractional groups $PSL(2,q),$ where $q$ is an odd prime power.

\subsection*{Declarations}
The first author is supported by the School of Mathematics at the University of Birmingham. The second author is member of the national research group GNSAGA. The authors have no conflicts of interest to declare that are relevant to the content of this article. Our research did not generate or reuse research data

\bibliographystyle{alpha}
\bibliography{references}

\newcommand{\etalchar}[1]{$^{#1}$}
\begin{thebibliography}{BFM{\etalchar{+}}18}

\bibitem[BFM{\etalchar{+}}18]{boileau-et-al}
Michel Boileau, Clara Franchi, Mattia Mecchia, Luisa Paoluzzi, and Bruno
  Zimmermann.
\newblock Finite group actions on 3-manifolds and cyclic branched covers of
  knots.
\newblock {\em J. Topol.}, 11(2):283--308, 2018.

\bibitem[Bre]{bredon1972introduction}
Glen~E Bredon.
\newblock {\em Introduction to compact transformation groups}.
\newblock (Academic press, New York, 1972).

\bibitem[CCN{\etalchar{+}}85]{Atlas}
J.~H. Conway, R.~T. Curtis, S.~P. Norton, R.~A. Parker, and R.~A. Wilson.
\newblock {\em {$\Bbb{ATLAS}$} of finite groups}.
\newblock Oxford University Press, Eynsham, 1985.
\newblock Maximal subgroups and ordinary characters for simple groups, With
  computational assistance from J. G. Thackray.

\bibitem[CL00]{cooper-long}
D.~Cooper and D.~D. Long.
\newblock Free actions of finite groups on rational homology {$3$}-spheres.
\newblock {\em Topology Appl.}, 101(2):143--148, 2000.

\bibitem[CS03]{conway-smith}
John~H. Conway and Derek~A. Smith.
\newblock {\em On quaternions and octonions: their geometry, arithmetic, and
  symmetry}.
\newblock A K Peters, Ltd., Natick, MA, 2003.

\bibitem[DH81]{Dotzel-Hamrick}
Ronald~M. Dotzel and Gary~C. Hamrick.
\newblock {$p$}-group actions on homology spheres.
\newblock {\em Invent. Math.}, 62(3):437--442, 1981.

\bibitem[GLS98]{GLS3}
Daniel Gorenstein, Richard Lyons, and Ronald Solomon.
\newblock {\em The classification of the finite simple groups. {N}umber 3.
  {P}art {I}. {C}hapter {A}}, volume~40 of {\em Mathematical Surveys and
  Monographs}.
\newblock American Mathematical Society, Providence, RI, 1998.
\newblock Almost simple $K$-groups.

\bibitem[Mac]{macwilliams19702}
Anne~R MacWilliams.
\newblock {\em On 2-groups with no normal abelian subgroups of rank 3, and
  their occurrence as Sylow 2-subgroups of finite simple groups}.
\newblock Transactions of the American Mathematical Society \textbf{150}(2)
  (1970) 345-408.

\bibitem[McC02]{mccullough}
Darryl McCullough.
\newblock Isometries of elliptic 3-manifolds.
\newblock {\em J. London Math. Soc. (2)}, 65(1):167--182, 2002.

\bibitem[Mec20]{mecchia2020finite}
Mattia Mecchia.
\newblock Finite groups acting on hyperelliptic 3-manifolds.
\newblock {\em J. Knot Theory Ramifications}, 29(4):2050021, 15, 2020.

\bibitem[MR02]{mecchia-reni2002}
Mattia Mecchia and Marco Reni.
\newblock Hyperbolic 2-fold branched coverings of links and their quotients.
\newblock {\em Pacific J. Math.}, 202(2):429--447, 2002.

\bibitem[MS15]{mecchia-seppi}
Mattia Mecchia and Andrea Seppi.
\newblock Fibered spherical 3-orbifolds.
\newblock {\em Rev. Mat. Iberoam.}, 31(3):811--840, 2015.

\bibitem[MZ06]{mecchia2006finite}
Mattia Mecchia and Bruno Zimmermann.
\newblock On finite simple groups acting on integer and mod 2 homology
  3-spheres.
\newblock {\em J. Algebra}, 298(2):460--467, 2006.

\bibitem[Ren00]{reni2000}
Marco Reni.
\newblock On {$\pi$}-hyperbolic knots with the same 2-fold branched coverings.
\newblock {\em Math. Ann.}, 316(4):681--697, 2000.

\bibitem[RZ]{reni2001finite}
Marco Reni and Bruno Zimmermann.
\newblock {\em Finite simple groups acting on 3-manifolds and homology
  spheres}.

\bibitem[Sak]{sakuma1995homology}
Makoto Sakuma.
\newblock {\em Homology of abelian coverings of links and spatial graphs}.
\newblock Canadian Journal of Mathematics \textbf{47}(1) (1995) 201-224.

\bibitem[Suz]{suzuki1986group}
M.~Suzuki.
\newblock {\em Group Theory. II,}.
\newblock Grundlehren der Mathematischen Wissenschaften, Vol. 248.
  (Springer-Verlag, New York, 1986).

\bibitem[Zim15]{zimmermann2015}
Bruno~P. Zimmermann.
\newblock About three conjectures on finite group actions on 3-manifolds.
\newblock {\em Sib. \`Elektron. Mat. Izv.}, 12:955--959, 2015.

\end{thebibliography}

\end{document}